\documentclass[12pt,nosumlimits,nonamelimits]{amsart}
\usepackage{amssymb,mathtools,accents}
\allowdisplaybreaks

\newtheorem{thmA}{Theorem}

\newtheorem{corA}[thmA]{Corollary}
\newtheorem{propA}[thmA]{Proposition}
\newtheorem{thm}{Theorem}
\newtheorem{prop}[thm]{Proposition}
\newtheorem{cor}[thm]{Corollary}
\newtheorem{lem}[thm]{Lemma}
\theoremstyle{remark}
\newtheorem{rem}[thm]{Remark}
\theoremstyle{definition}
\newtheorem{defn}[thm]{Definition}
\newtheorem*{conj}{Conjecture}
\newcommand{\IB}{\mathbb{B}}
\newcommand{\IC}{\mathbb{C}}
\newcommand{\IH}{\mathbb{H}}

\newcommand{\IM}{\mathbb{M}}
\newcommand{\IN}{\mathbb{N}}
\newcommand{\IQ}{\mathbb{Q}}
\newcommand{\IR}{\mathbb{R}}

\newcommand{\IZ}{\mathbb{Z}}

\newcommand{\bF}{\mathbf{F}}

\newcommand{\fS}{\mathfrak{S}}

\newcommand{\rf}{\mathrm{f}}
\newcommand{\rC}{\mathrm{C}}

\newcommand{\rr}{\mathrm{r}}
\newcommand{\rE}{\mathrm{E}}

\newcommand{\cC}{\mathcal{C}}

\newcommand{\cF}{\mathcal{F}}
\newcommand{\cG}{\mathcal{G}}
\newcommand{\cH}{\mathcal{H}}

\newcommand{\cM}{\mathcal{M}}

\newcommand{\cR}{\mathcal{R}}
\newcommand{\cS}{\mathcal{S}}

\newcommand{\cU}{\mathcal{U}}
\newcommand{\cV}{\mathcal{V}}
\newcommand{\cZ}{\mathcal{Z}}

\newcommand{\ve}{\varepsilon}
\newcommand{\vp}{\varphi}
\newcommand{\Ga}{\Gamma}
\newcommand{\La}{\Lambda}
\newcommand{\acts}{\curvearrowright}
\newcommand{\id}{\mathrm{id}}

\newcommand{\UE}{\mathcal{UE}}
\newcommand{\fin}{\mathrm{fin}}
\DeclareMathOperator{\Aut}{Aut}
\DeclareMathOperator{\Ad}{Ad}
\DeclareMathOperator{\cb}{cb}
\DeclareMathOperator{\dom}{dom}
\DeclareMathOperator{\Prob}{Prob}
\DeclareMathOperator{\Lim}{Lim}

\DeclareMathOperator{\rank}{rank}
\DeclareMathOperator{\supp}{supp}
\DeclareMathOperator{\GL}{GL}
\DeclareMathOperator{\SL}{SL}
\DeclareMathOperator{\dist}{d}
\DeclareMathOperator{\conv}{conv}
\newcommand{\lspan}{\mathop{\mathrm{span}}}
\newcommand{\ip}[1]{\mathopen{\langle}#1\mathclose{\rangle}}
\def\bigast{\font\bigsymbolsfont=cmsy10 scaled \magstep3
 \setbox0=\hbox{\bigsymbolsfont\char'003 }\mathord{\lower1pt\box0}}\relax\ignorespaces
\newcommand{\cst}{$\rC^*$-alge\-bra}
\title{Uniform amenability at infinity}
\author{Narutaka Ozawa}
\address{RIMS, Kyoto University, \mbox{606-8502} Japan}
\email{narutaka@kurims.kyoto-u.ac.jp}
\thanks{The author was partially supported by JSPS KAKENHI Grant Numbers 24K00527, 25H00588, 25H00593}

\date{\today}

\begin{document}
\begin{abstract}
We introduce the notion of uniform exactness, or 
uniform amenability at infinity, for discrete groups and 
prove it for a wide class of groups containing free groups 
and their limit groups. 
This shows a novel strong convergence phenomenon 
that any convergent sequence of such groups 
in the space of marked groups converges strongly 
in the operator algebraic sense. In particular, 
convergence of the spectral radius formula 
is uniform over probability measures on such groups 
whose supports have a fixed cardinality.
\end{abstract}
\maketitle
\section{Introduction}
Amenability is one of the most important notions in analytic group theory. 
However, many groups, in particular those groups that contain non-cyclic free groups, 
are not amenable. 
It is a remarkable fact that many non-amenable groups still enjoy 
the weaker yet very useful property, 
\emph{exactness} or a.k.a.\ \emph{amenability at infinity} 
(\cite{anantharaman-delaroche, kw}). 
A group $\Ga$ is exact if the $\Ga$-action 
on the Stone--\v{C}ech compactification $\beta\Ga$, or equivalently 
on the \cst{} $\ell_\infty\Ga$ of bounded functions on $\Ga$, 
is amenable. 
Here the action $\Ga\acts\ell_\infty\Ga$ is given by 
$(sf)(x) \coloneq f(s^{-1}x)$ for $s,x\in\Ga$ and $f\in\ell_\infty\Ga$. 
For example, the class of exact groups contains 
hyperbolic groups (\cite{adams}) 
and linear groups (\cite{ghw}), 
but not all groups are exact (\cite{ad,gromov,osajda}). 
Both amenability and exactness are inherited by subgroups. 
See \cite{adr, bo} for a comprehensive treatment on 
this topic. 

A group $\Ga$ is said to be \emph{uniformly amenable} 
(\cite{bozejko, keller, ks, wysoczanski}) if 
its ultrapower $\Ga^\cU$ is amenable. 
Here and throughout the paper, we fix 
a free ultrafilter $\cU$ on $\IN$. 
All of the arguments do not depend on the choice of a free ultrafilter $\cU$.
The ultrapower $\Ga^\cU$ is defined to be $\prod_{\IN}\Ga/{\sim}$, 
where $(s_n)_n\sim(t_n)_n$ if $\{ n : s_n=t_n \} \in \cU$. 
The element in the ultrapower that arises from 
a sequence $(x_n)_n$ etc will be written by $[x_n]_n$.
For example, solvable groups are uniformly amenable 
since being solvable of derived length less than a fixed number is 
maintained by ultraproduct. On the other hand, 
the amenable group $\fS_\infty$ of finite permutations 
on $\IN$ is not since it contains a non-cyclic free group 
in the ultrapower.
It is tempting to introduce the notion of uniform exactness 
similarly by exactness of the ultrapower, 
but this appears too weak for meaningful applications. 
Instead, we define a group $\Ga$ to be \emph{uniformly exact} 
(or \emph{uniformly amenable at infinity}) if 
the $\Ga^\cU$-action on $(\ell_\infty\Ga)^\cU$ 
is amenable.
This notion will have applications 
to a strong convergence phenomenon 
and a uniform estimate of the spectral radius formula.  
Recall that for a \cst{} $A$, its ultrapower is defined to be 
\[
A^\cU \coloneq \{ (a_n)_n \in\ell_\infty(\IN,A) \}/\{ (a_n)_n : \Lim_\cU \|a_n\|=0\}.
\]
Here $\Lim_{\cU}$ is the character on $\ell_\infty\IN$ associated 
with the free ultrafilter $\cU$. 
We view $(\ell_\infty\Ga)^\cU$ 
as the $\mathrm{C}^*$-sub\-alge\-bra of ``internal'' functions 
in $\ell_\infty(\Ga^\cU)$ 
via the $\Ga^\cU$-equivariant isometric embedding given by 
$[f_n]_n([x_n]_n) \coloneq \Lim_\cU f_n(x_n)$. 
Since the condition for amenability gets more restrictive  
as the relevant \cst{} gets smaller, 
uniform exactness of $\Ga$ implies exactness of its 
ultrapower $\Ga^\cU$. 
It is known that $\Ga^\cU$ is exact 
for torsion-free hyperbolic groups (\cite{km2,sela:plms})
and linear groups (\cite{ghw}). 
Besides uniformly amenable groups, these 
are the main source (and the only source as far as 
the author is aware) of groups that have exact ultrapowers.
\begin{conj}
Torsion-free hyperbolic groups and linear groups are uniformly exact. 
\end{conj}
We solve this conjecture for the case of free groups $\bF$, 
by viewing the ultrapower $\bF^\cU$ as a $\IZ^\cU$-tree 
(see \cite{chiswell}) and studying the ``definable'' functions 
on $\bF^\cU$ through the tree compactification. 
\begin{thmA}\label{thmA:freegroup}
Free groups are uniformly exact. 
\end{thmA}

Let $\lambda$ denote the left regular representation 
of $\Ga$ on $\ell_2\Ga$. 
The \emph{reduced group \cst{}} 
$\rC^*_\rr\Ga$ of $\Ga$ is the norm closure of 
$\lambda(\IC\Ga)$ in the \cst{} $\IB(\ell_2\Ga)$ 
of all bounded linear operators on $\ell_2\Ga$. 
It is well-known that exactness has $\mathrm{C}^*$-alge\-braic 
characterization (\cite{kw}, Theorem 5.1.7 in \cite{bo}). 
Namely, a group $\Ga$ is exact if and only if 
$\rC^*_\rr\Ga$ is \emph{exact} as a \cst{}. 
Note that the ultrapower \cst{} 
$(\rC^*_\rr\Ga)^\cU$ is exact if and only if 
it is \emph{subhomogeneous} if and only if 
$\Ga$ is virtually abelian. 
This shows a stark contrast between 
exactness of $\rC^*_\rr(\Ga^\cU)$ and that of $(\rC^*_\rr\Ga)^\cU$.

\begin{thmA}\label{thmA:unifamen}
A group $\Ga$ is uniformly exact if and only if 
the canonical embedding 
\[
A^\cU \otimes \rC^*_\rr(\Ga^\cU)
 \hookrightarrow (A \otimes \rC^*_\rr\Ga)^\cU,\quad 
 [a_n]_n\otimes\lambda([t_n]_n)\mapsto [a_n\otimes\lambda(t_n)]_n
\]
is continuous and isometric for every \cst{} $A$. 
Here the tensor product $\otimes$ is 
the minimal (a.k.a.\ spatial) tensor product of \cst{}s. 
\end{thmA}

A group $\Ga$ is said to be 
\emph{mixed-identity-free} or MIF (\cite{ho}) 
if there is a nontrivial element in $\Ga^\cU$ 
that is free from $\Ga$, or equivalently 
if there is an embedding 
$\Ga\ast\IZ \hookrightarrow \Ga^\cU$ 
that extends the diagonal embedding of $\Ga$. 
It has been proved in \cite{agkp, selfless, robert} 
that \emph{some} of these 
embeddings for some of the MIF groups 
are extended to the 
\cst{} embeddings 
$\rC^*_\rr(\Ga\ast\IZ) \hookrightarrow (\rC^*_\rr\Ga)^\cU$, 
proving that such $\rC^*_\rr\Ga$ are selfless. 
Theorem~\ref{thmA:unifamen} states that 
\emph{all} of these embeddings are extended, 
provided that $\Ga$ is uniformly exact. 
In particular, a uniformly exact group is 
\emph{completely $\mathrm{C}^*$-self\-less} 
(\cite{agkp, selfless, robert})
if it is MIF.

Recall that a \emph{marked group} of rank $d$ 
is a group $\Ga$ together with 
an epimorphism $\vp$ 
from the rank $d$ free group $\bF_d$ 
(with the distinguished generating set) onto $\Ga$. 
A marked group (or more precisely, its 
marked isomorphism class) 
is in the one-to-one correspondence to 
a normal subgroup $\ker\vp$ in $\bF_d$. 
The space $\cM_d$ of marked groups of rank $d$ 
is a compact topological space as a closed subset of $2^{\bF_d}$. 
The space $\cM$ of marked groups is the disjoint 
union $\bigsqcup_d\cM_d$. 
A \emph{limit group} (\cite{km,sela}) 
is a group that appears as the limit in $\cM$ 
of marked groups that are isomorphic to free groups. 
We say that a subgroup-closed class $\cC$ of groups 
is \emph{compact} if the image of $\cC$ 
in $\cM$ (under all possible markings) is closed. 
The class of limit groups is compact by definition. 
Every uniformly exact group has a function $k\colon\IN\to\IN$ as 
a \emph{modulus of uniform exactness} (see Section~\ref{sec:mod}). 
The class $\UE$ of uniformly exact groups 
is the union of uniformly exact groups $\UE(k)$ having 
$k$ as a modulus of uniform exactness.

\begin{corA}\label{corA:list}
The class $\UE = \bigcup_{k\in\IN^\IN} \UE(k)$ satisfies the following. 
\begin{enumerate}
\item
The class $\UE$ contains free groups and more generally 
torsion-free hyperbolic groups having tame geometry at infinity 
(see Section~\ref{sec:tame}).
\item 
An amenable group is uniformly exact if and only if 
it is uniformly amenable. 
Moreover, for every $k\in\IN^\IN$, the class of amenable 
groups in $\UE(k)$ is compact. 
In particular, $\fS_\infty$ is not uniformly exact. 
\item
For every $k\in\IN^\IN$, 
the class $\UE(k)$ is compact and is closed under 
taking subgroups, directed unions, 
and quotients by normal amenable subgroups. 
In particular, limit groups are uniformly exact with the same 
modulus of uniform exactness as free groups.
\item\label{cond:extension}
The class $\UE$ is closed under extensions, free products, graph products over finite graphs, and finite-index super\-groups.
More precisely, for every $k\in\IN^\IN$, there is $k'\in\IN^\IN$ 
such that extensions and free products 
of two groups from $\UE(k)$ are contained in $\UE(k')$. 
\end{enumerate}
\end{corA}

Exactness is closed under free products over 
arbitrary amalgamation and under HNN extensions 
(\cite{dykema}, see also Theorem 5.2.7 in \cite{bo}). 
The same for uniform exactness is much unclear. 

The following proposition strengthens the result of \cite{lm} 
that every limit group is a strong limit of \emph{some} 
sequence of free groups. 
See \cite{magee, vanhandel} for recent surveys on 
strong convergence phenomena. 

\begin{propA}\label{propA:norm}
Fix a modulus $k\in\IN^\IN$ of uniform exactness. 
Then for every $\Ga\in\UE(k)$, $f\in\IC\Ga$, and $\ve>0$ 
with $F\coloneq \supp f$ and 
$E \coloneq (F\cup\{1\}\cup F^{-1})^{k(|F|+\lfloor2/\ve\rfloor)}$, 
one has 
\[
\| \lambda(f) \|_{\rC^*_\rr\Ga}
 \le (1+\ve) \| P_{\ell_2E} \lambda(f) |_{\ell_2 E} \|_{\IB(\ell_2E)}. 
\]
In particular, every convergent sequence 
$(\vp_n\colon \bF_d\to\Ga_n)_n$ in $\cM_d$ 
with $\Ga_n\in\UE(k)$ converges to 
the limit $\vp_\infty\colon\bF_d\to\Ga_\infty$ 
strongly in the sense that 
\[
\forall f\in\IC\bF_d\quad
\lim_n \|\lambda(\vp_n(f))\|_{\rC^*_\rr\Ga_n}
 = \|\lambda(\vp_\infty(f))\|_{\rC^*_\rr\Ga_\infty}. 
\]
Also, convergence of the spectral radius formula
\[
\|\lambda(f) \|_{\rC^*_\rr\Ga} = \lim_n \sqrt[2n]{(f^**f)^{*n}(1)}.
\]
is uniform over $\Ga \in \UE(k)$ and over $f \in\IC\Ga$ 
with uniform bounds on $|\supp f|$ and $\|f\|_1$. 
\end{propA}

\subsection*{Acknowledgments} 
The author got the key idea for this research 
during his stay at the Isaac Newton Institute 
for the program 
``Operators, Graphs, Groups'' in October 2025. 
He acknowledges the kind hospitality and the exciting
environment provided by the institute. 
He thanks M.~Bestvina, Y.~Cornulier, K.~Fujiwara, D.~Gaboriau, M.~Magee, 
F.~Paulin, M.~de la Salle, and Z.~Sela for fruitful and stimulating conversations. 
He also thanks M.~ Bo\.{z}ejko for persistently asking him 
what uniform exactness should be about twenty years ago.
He is particularly grateful to de la Salle for pointing out 
an error in the earlier draft and suggesting 
an improvement of Theorem~\ref{thmA:LaTree}. 
Theorem~\ref{thm:tame} for the case of the field $\IR$ 
was provided to us by GPT-5.6, but was also 
independently observed by de la Salle. 
\section{Main theorem for actions on $\La$-trees}
The following is the main technical result of this paper. 
The terminologies used in the statement will be explained eventually. 

\begin{thm}\label{thmA:LaTree}
Let $\Ga$ be a countable group, 
$\La$ be an ordered abelian group, 
and $X$ be a $\La$-tree 
on which $\Ga$ acts isometrically without inversions. 
Assume that the action $\Ga\acts X$ satisfies 
the following conditions. 
\begin{enumerate}
\item\label{cond:LaTreeCA}
The set of amenable subgroups of $\Ga$ is countable.
\item\label{cond:LaTreeAS}
For every $x,y\in X$, $x\neq y$,  
the subgroup 
\[
\{ t\in\Ga : \dist(tx,x)+\dist(ty,y) \ll \dist(x,y) \mbox{ in }\La\}
\]
is amenable. 
\item\label{cond:LaTreeFR} 
The action $\Ga\acts X$ has finite $\La$-rank 
$\rank_\La(\Ga\acts X)$. 
\end{enumerate}
Let $K$ be a compact $\Ga$-space. 
Assume that $K$ is amenable 
as a compact $\Ga^x$-space for the stabilizer subgroup 
$\Ga^x \coloneq \{ t\in \Ga : tx=x\}$ at every $x\in X$ 
and that the tree compactification $\hat{X}$ of $X$ is 
a $\Ga$-equivariant quotient of $K$. 
Then the compact $\Ga$-space $K$ is amenable. 
\end{thm}
\section{Amenability of $\Ga$-spaces and equivalence relations}\label{sec:amen}
Let $\Ga$ be a group and denote by 
$\Prob\Ga \subset \ell_1\Ga$ 
the convex subset of probability measures on $\Ga$. 
The space $\Prob\Ga$ is 
equipped with the norm topology 
(which coincides with the pointwise convergence topology) and with 
the $\Ga$-action by left translations. 
Recall that a group $\Ga$ is \emph{amenable} if there is 
a net $\xi_i\in\Prob\Ga$, 
called an \emph{approximate invariant mean}, that 
satisfies 
$\lim_i \|\xi_i - s\xi_i\|_1=0$ for every $s\in\Ga$. 
We review the notion of amenability that is generalized 
for $\Ga$-actions (\cite{anantharaman-delaroche,adr}) 
and equivalence relations (\cite{jkl}) through 
existence of appropriate approximate invariant means.

A \emph{compact $\Ga$-space} $K$ 
is a compact topological space $K$ 
equipped with an action of $\Ga$ by homeomorphisms, 
and a \emph{Borel $\Ga$-space} $Z$ is a 
standard Borel space equipped 
with an action of a countable group 
$\Ga$ by Borel isomorphisms. 
Be aware that we put a countability assumption 
on Borel $\Ga$-spaces, but not on compact $\Ga$-spaces. 

\begin{defn}\label{def:topoamen}
A compact $\Ga$-space $K$ is \emph{(topologically) amenable} 
if there is a net 
of continuous maps $\eta_i\colon K\to\Prob\Ga$ that satisfies 
\[
\forall s\in\Ga\quad \lim_i \sup_{x\in K} \| \eta_i^{sx} - s\eta_i^x \|_1 = 0. 
\] 
By continuity of $\eta_i$, we may assume that 
$E_i \coloneq \bigcup_x \supp\eta_i^x$ is a finite subset of $\Ga$ 
(see Lemma 4.3.8 in \cite{bo}). 
The net $(\eta_i)_i$ is called an \emph{approximate equivariant mean}. 
Instead of saying that $K$ is amenable, 
we also say that the action $\Ga\acts K$ 
is \emph{topologically amenable} 
or the $\Ga$-action on the \cst{} $C(K)$ is amenable. 

Let $\Prob Z$ denote the set of probability measures on $Z$. 
A Borel $\Ga$-space $Z$ (with $\Ga$ being countable) 
is \emph{amenable} if there is a net of 
Borel maps $\eta_i\colon Z\to\Prob \Ga$ 
that satisfies 
\[
\forall s\in\Ga\ \forall m\in \Prob Z\quad \lim_i \int_Z \| \eta_i^{sx} - s\eta_i^x \|_1\,dm(x) = 0, 
\]
or equivalently, if for every finite 
subset $F\Subset\Ga$, $m\in \Prob Z$, and $\ve>0$, 
there is a Borel map $\eta\colon Z\to\Prob \Ga$ that satisfies
\[
\sum_{s\in F}\int_Z \| \eta^{sx} - s\eta^x \|_1\,dm(x) < \ve. 
\]
Here $\eta\colon Z\to\Prob \Ga$ is Borel if and only if 
$x\mapsto\eta^x(t)$ is Borel for every $t\in\Ga$.
\end{defn}

We make some observations. 
For a compact $\Ga$-space $K$ and its 
quotient $\Ga$-space $L$   
(or equivalently $C(L)\subset C(K)$),  
amenability of $L$ is succeeded by $K$. 
A compact/Borel $\Ga$-space is amenable if 
it is the case for the restriction to 
finitely generated subgroups. 
A countable union of amenable Borel $\Ga$-spaces 
is amenable. 
The following is very useful. 

\begin{thm}[\cite{anantharaman-delaroche,adr}, see also 5.2.1 in \cite{bo}]\label{thm:topbob}
Let $\Ga$ be a countable group 
and $K$ be a second countable compact $\Ga$-space. 
Then $K$ is topologically amenable if and only if it is amenable 
as a Borel $\Ga$-space. 
\end{thm}

Let $Z$ be a standard Borel space. 
A \emph{countable Borel equivalence relation} $\cR$ on $Z$ 
is an equivalence relation $\cR\subset Z\times Z$ such that 
each equivalence class is at most countable and that $\cR$ 
is a Borel subset of $Z\times Z$. 
For $x\in Z$, we write $[x]_{\cR}$ 
for the equivalence class that contains $x$.
We denote by $[\cR]$ the \emph{full group} consisting 
of Borel bijections $\vp$ on $Z$ that 
satisfies $\vp(x)\in[x]_\cR$ for all $x$.
For a function $\nu$ on $\cR$, 
we denote by $\nu^x$ the function on $[x]_{\cR}$ 
given by $\nu^x(y) \coloneq \nu(x,y)$. 
We need a notion of amenability for countable Borel 
equivalence relations that is stronger than 
the standard one (\cite{jkl}). 
Although Fr\'{e}chet-amenability (\cite{jkl}) meets 
our purpose, we introduce an even stronger one. 

\begin{defn}
A countable Borel equivalence relation $\cR$ on $Z$ 
is \emph{$\cM$-amenable} if there is a net of 
non-negative Borel functions $\nu_i$ on $\cR$ 
such that $\nu_i^x\in\Prob {[x]_{\cR}}$ 
for all $x$ and 
\[
\forall \vp\in[\cR]\,\forall m\in\Prob Z\quad 
\lim_i\int_Z \|\nu_i^{\vp(x)}-\nu_i^x\|_1 \,dm(x)=0.
\]
\end{defn}

For a Borel $\Ga$-space $Z$ (with $\Ga$ being countable), 
we denote the corresponding 
\emph{orbit equivalence relation} by 
\[
\cR_{\Ga\acts Z} \coloneq \{ (x,t^{-1}x) : x\in Z,\ t\in\Ga \} \subset Z\times Z.
\]
By the Feldman--Moore theorem, 
every countable Borel equivalence relation arises in this way.
If the Borel action $\Ga\acts Z$ is free, then the bijective 
correspondence 
$Z\times\Ga \ni (x,t) \leftrightarrow (x,t^{-1}x) \in \cR_{\Ga\acts Z}$ 
gives equivalence between amenability of $\Ga\acts Z$ and 
that of $\cR_{\Ga\acts Z}$. 
We note that $\cR_{\Ga\acts Z}$ is $\cM$-amenable 
whenever $\Ga$ is amenable. 
We will collect basic permanence properties of amenability. 

\begin{lem}\label{lem:fd} 
Let $\Ga$ be a countable group and 
$Z$ be a Borel $\Ga$-space. 
Suppose $\Delta\le\Ga$ is a subgroup and $Y\subset Z$ is 
a $\Delta$-invariant Borel subset that satisfy 
$gY\cap Y=\emptyset$ for $g\in\Ga\setminus\Delta$ and  
$\Ga Y=Z$.  
If $Y$ is amenable as a Borel $\Delta$-space, then $Z$ is amenable. 
\end{lem}
\begin{proof}
We view $\bar{\Ga} \coloneq \Ga/\Delta$ 
as a space on which $\Ga$ acts; 
$\Ga\times\bar{\Ga}\ni(s,p)\mapsto sp\in\bar{\Ga}$.
Fix a lift $\varsigma\colon \bar{\Ga}\to\Ga$ and let 
$\alpha\colon \Ga \times \bar{\Ga} \to \Delta$ be 
the corresponding cocycle given by 
$\alpha(s,p) \coloneq \varsigma(sp)^{-1}s\varsigma(p)$. 
We may assume $Z=\bar{\Ga}\times Y$, via 
the Borel bijection 
$\bar{\Ga}\times Y\ni (p,y) \mapsto \varsigma(p)y\in Z$, 
where 
$\Ga$ acts on $\bar{\Ga}\times Y$ by $s(p,y) = (sp,\alpha(s,p)y)$.
Given an approximate equivariant mean $\eta_i\colon Y\to \Prob \Delta$, 
we define $\zeta_i\colon \bar{\Ga}\times Y\to \Prob \Ga$ 
by $\zeta_i^{(p,y)} \coloneq \varsigma(p)\eta_i^y$. 
Then, for every $s\in\Ga$, 
\[
\| \zeta_i^{s(p,y)} - s\zeta_i^{(p,y)} \|_1
 = \| \varsigma(sp)\eta_i^{\alpha(s,p)y} - s\varsigma(p)\eta_i^y \|_1
 = \| \eta_i^{\alpha(s,p)y} -  \alpha(s,p)\eta_i^y \|_1.
\]
Hence, for every $m =(m_p)_{p\in\bar{\Ga}}\in \Prob(\bar{\Ga}\times Y)$, one has 
\[
\int \| \zeta_i^{s(p,y)} - s\zeta_i^{(p,y)} \|_1\,dm(p,y)
 = \sum_p \int \| \eta_i^{\alpha(s,p)y} -  \alpha(s,p)\eta_i^y \|_1\,dm_p(y)
 \to0,
\]
since each integral in the RHS is bounded by $2\|m_p\|$ and 
converges to $0$. 
\end{proof}

\begin{lem}\label{lem:borelamencriterion}
Let $\Ga$ be a countable group and 
$Z$ be a Borel $\Ga$-space. 
Suppose that the countable Borel equivalence 
relation $\cR_{\Ga\acts Z}$ is $\cM$-amenable 
and that $\{\Ga^x : x\in Z\}$ consists of 
a countable family of amenable subgroups. 
Then the Borel $\Ga$-space $Z$ is amenable.
\end{lem}
\begin{proof}
We may assume by a countable decomposition 
that all $\Ga^x$ are conjugate to a single subgroup $\Delta\le\Ga$. 
The Borel subset $Y \coloneq \{ x\in Z : \Ga^x =\Delta \}$ 
is invariant under the normalizer subgroup $N_\Ga(\Delta)$ 
of $\Delta$ in $\Ga$.
Thanks to Lemma~\ref{lem:fd}, 
it suffices to show that $Y$ is amenable as 
a Borel $N_\Ga(\Delta)$-space.
Thus, we may assume that $Y=Z$ and $N_\Ga(\Delta)=\Ga$, 
i.e., $\Delta = \ker(\Ga\acts Z)$ 
and the action $\bar{\Ga} \coloneq \Ga/\Delta\acts Z$ is free. 
It follows that $Z$ is amenable as a Borel $\bar{\Ga}$-space. 

As in Proof of Lemma~\ref{lem:fd},  
fix a lift $\varsigma\colon \bar{\Ga}\to\Ga$ and the cocycle
$\alpha\colon \Ga \times \bar{\Ga} \to \Delta$.
Thus the bijection
$\theta\colon \bar{\Ga} \times \Delta \to\Ga$, given by 
 $(p,t) \mapsto \varsigma(p)t$, 
satisfies $s\theta(p,t) = \theta( sp,\alpha(s,p)t )$. 
For $\eta\colon Z\to \Prob \bar{\Ga}$ 
and $\xi\in\Prob \Delta$, we define 
$\zeta\colon Z\to\Prob \Ga$ by 
$\zeta^x = (\eta^x\otimes \xi)\circ\theta^{-1}$. 
Then 
\begin{align*}
\| \zeta^{sx}-s\zeta^x\|_1
 &\le \| \zeta^{sx}-(s\eta^x\otimes\xi)\circ\theta^{-1}\|_1
   + \| (s\eta^x\otimes\xi)\circ\theta^{-1}- s\zeta^x\|_1\\
 &\le \| \eta^{sx}-s\eta^x\|_1 
   + \sum_p\eta^x(s^{-1}p)\|\xi-\alpha(s,s^{-1}p)\xi\|_1.
\end{align*}
Let a finite subset $F\Subset\Ga$, $m\in\Prob Z$, and $\ve>0$ be given. 
There is $\eta$ that satisfies 
\[
\sum_{s\in F}\int_Z \| \eta^{sx}-s\eta^x\|_1\,dm(x) < \ve/2.
\]
Since $\sum_{s,p}\int_Z \eta^x(s^{-1}p)\,dm(x) = |F|$, 
there is $\xi$ that satisfies 
\[
\sum_{s\in F}\int_Z \sum_p\eta^x(s^{-1}p)\|\xi-\alpha(s,s^{-1}p)\xi\|_1\,dm(x) <\ve/2.
\]
It follows that for these $\eta$ and $\xi$, 
the Borel map $\zeta\colon Z\to\Prob \Ga$ satisfies\\ 
\begin{minipage}{.9\textwidth}
\[
\sum_{s\in F}\int_Z \| \zeta^{sx}-s\zeta^x\|_1\,dm(x) < \ve.
\]
\end{minipage} 
\end{proof}

\begin{prop}\label{prop:amencriterion}
Let $\Ga$ be a countable group and 
$K$ be a compact $\Ga$-space 
with a $\Ga$-equi\-variant quotient map $Q$ from 
$K$ onto a second countable compact $\Ga$-space $Y$. 
Suppose that there is a countable $\Ga$-invariant subset 
$B\subset Y$ such that the Borel $\Ga$-space 
$Y\setminus B$ is amenable and that 
$Q^{-1}(b)$ is amenable as a compact $\Ga^b$-space 
for every $b\in B$.
Then the compact $\Ga$-space $K$ is amenable.  
\end{prop}
\begin{proof}
Since topological amenability is witnessed by 
a countable family of continuous functions, 
we may assume that $K$ is second countable. 
The Borel $\Ga$-space
$Q^{-1}(Y\setminus B)$ is amenable and, by Lemma~\ref{lem:fd}, 
so are $Q^{-1}(\Ga b)$ for $b\in B$. 
Thus $K$ is a countable union of amenable Borel $\Ga$-spaces 
and hence is topologically amenable 
by Theorem~\ref{thm:topbob}. 
\end{proof}

The following lemmas are taken from \cite{jkl}.

\begin{lem}[\cite{jkl}]
Let $\cR$ be a countable Borel equivalence relation on $Z$ 
and $A\subset Z$ be a Borel subset. 
If $\cR$ is $\cM$-amenable, 
then so is the restriction $\cR|_A \coloneq \cR\cap(A\times A)$.  
\end{lem}
\begin{proof}
There is a Borel map $\pi\colon Z\to A$ 
such that $(\pi(x),x)\in\cR$ for every $x$. 
If $\nu_i$ is an approximate invariant mean on $\cR$, 
then $(\id\times\pi)_*(\nu_i)$ is an approximate invariant mean on $\cR|_A$.
\end{proof}

\begin{lem}[\cite{jkl}]
Let $\cS$ and $\cR$ be countable Borel equivalence relations 
on $Z$ such that $\cS\subset\cR$. 
If $\cR$ is $\cM$-amenable, then so is $\cS$.
\end{lem}
\begin{proof}
We may assume that $\cR=\cR_{\Ga\acts Z}$ 
is the Borel orbit equivalence relation of $\Ga\acts Z$.  
Let's enumerate $\Ga = \{ t_k : k \in \IN\}$ and 
define $\theta\colon\cR\to\cS$ by 
\[
\theta(x,y) \coloneq (x, t_l y),\mbox{ where } l \coloneq \min\{ k : (x,t_k y) \in\cS\}.
\]
Note that $\theta(x,y)=\theta(x',y)$ if $(x,x')\in\cS$. 
If $\nu_i$ is an approximate invariant mean on $\cR$, 
then $\theta_*(\nu_i)$ is an approximate invariant mean on $\cS$.
\end{proof}
\begin{lem}[\cite{jkl}]\label{lem:becd}
The union $\cR \coloneq \bigcup_n\cR_n$ of 
an increasing sequence of $\cM$-amenable 
countable Borel equivalence relations $\cR_n$ 
is $\cM$-amenable.
\end{lem}
\begin{proof}
This is where we need $\cM$-amenability rather than the plain amenability.
\end{proof}
\section{Real trees and the tree compactification}
An \emph{$\IR$-tree} (or a \emph{real tree}) $T$ 
is a nonempty geodesic metric space that is $0$-hyper\-bolic. 
See \cite{bestvina:survey,chiswell,evans} for details. 
The \emph{Gromov product} with respect 
to a point $u\in T$ is defined for $(x,y)\in T\times T$ by 
\[
\ip{x,y}_u \coloneq \frac{1}{2}(\dist(x,u)+\dist(y,u)-\dist(x,y)).
\] 
An \emph{arc}  (resp.\ \emph{tripod}) is a bounded complete 
subtree with exactly two (resp.\ three) endpoints. 
For every $x,y\in T$, $x\neq y$, there is a unique arc $[x,y]$ 
connecting $x$ and $y$. 
For every $x,y,u$, the center $w$ of tripod spanned by $\{x,y,u\}$ 
satisfies $\dist(w,u) = \ip{x,y}_u$. 
A point in $T$ is called a \emph{branch point} if it is a center of tripod. 
A \emph{ray} is the image in $T$ 
of an isometric embedding $\alpha$ 
of the interval $[0,d)\subset\IR$, 
where $d\in\IR_{>0}\cup\{\infty\}$, 
such that $\lim_{t\to d}\alpha(t)$ does not exist in $T$. 
The \emph{boundary} $\partial T$ is the space 
of the equivalence classes 
of rays, where two rays are \emph{equivalent} 
if their intersection is a ray. 
We write $\bar{T} \coloneq T\cup\partial T$.  
The Gromov product extends on $\bar{T}\times\bar{T}$. 
For $u\in T$ and $x\in\partial T$, we write $[u,x)$ 
for the unique ray that represents $x$ and 
has $u$ as the endpoint, and $[u,x] \coloneq [u,x)\cup\{x\}$. 
We also write $[x,y) \coloneq [x,y]\setminus\{y\}$ and 
$(x,y)\coloneq[x,y]\setminus\{x,y\}$, which is called an \emph{open arc}. 

We will introduce the tree compactification $\hat{T}$ of an $\IR$-tree, 
following the construction for a simplicial tree (see \cite{bowditch,ms} 
and Section 5.2 in \cite{bo}).
It will be shown that $\hat{T}$ coincides with $\bar{T}$ as a set of points. 
For every $u,v\in T$, we define the function $h_{u,v}$ 
on $T$ by $h_{u,v}(x) \coloneq \dist(u,v)^{-1}\ip{x,v}_u$ (or identically 
$0$ if $u=v$). 
We consider the $\rC^*$-sub\-alge\-bra $\cC(T)$ 
of $\ell_\infty T$ generated by $\{ h_{u,v} : u,v\in T\}$ and 
define the \emph{tree compactification} $\hat{T}$ of $T$ to be 
the Gelfand spectrum of the unital commutative 
\cst{} $\cC(T)$.
The point evaluation at $T$ gives rise 
to a dense embedding of $T$ into $\hat{T}$. 
Since every $h_{u,v}$ is continuous 
w.r.t.\ the original metric topology on $T$ 
and has limit along every ray $\alpha$, one has the following. 
\begin{lem}
The dense embedding of $T$ into $\hat{T}$ is continuous. 
Moreover, it extends to an embedding of $\bar{T}$ into $\hat{T}$. 
\end{lem}

The embedding of $T$ into $\hat{T}$ need not be homeomorphic as 
any net of points $x_i$ that exist in different directions relative to 
$x$ converges to $x$. 
For $x\in T$, we observe that $x_i\to x$ in $\hat{T}$ 
if and only if $\ip{x_i,u}_x\to0$ for every $u\in T$, since 
\[
|h_{u,v}(x_i) - h_{u,v}(x)| \le \ip{x_i,u}_x+\ip{x_i,v}_x. 
\]

\begin{thm} $\hat{T} = \bar{T}$. 
\end{thm}
\begin{proof}
We follow the proof of Proposition 5.2.5 in \cite{bo}.
It suffices to show that every net $(x_i)_{i\in I}$ 
in $T$ has a limit point (in the topology of $\hat{T}$) in $\bar{T}$.
Take a cofinal ultrafilter $\cV$ on $I$.
Pick a base point $o$ and consider the subset 
\[
\alpha \coloneq \{ z\in T : \{ i \in I : z \in  [o,x_i] \}\in\cV \}
\]
It is not hard to see that $\alpha$ is connected 
and does not contain a tripod. 
Hence $\alpha$ is either $[o,x)$ for $x\in \bar{T}$ 
or $[o,x]$ for $x\in T$ which is possibly degenerate 
when $x=o$. 
We claim that in either case $\Lim_\cV x_i = x$ in $\hat{T}$.  
We assume that $x\in T$ 
as the case for $x\in\partial T$ is easier. 
Thus it suffices to show that 
$\Lim_\cV\ip{x_i,u}_x=0$ for every $u\in T$. 
Suppose that this is not the case and  
there is $u\in T$ such that $\Lim_\cV\ip{x_i,u}_x>0$. 
Since $\alpha$ does not extend beyond $x$, this implies 
$\ip{o,u}_x > 0$. 
Thus $\ve \coloneq \Lim_\cV \ip{o,x_i}_x >0$ by $0$-hyper\-bolicity. 
Then $z \in \alpha$ with $\dist(z,x) < \ve$ 
if off $[o,x_i]$ for $\cV$-many $i$, a contradiction. 
\end{proof}

\begin{lem}\label{lem:treeemb1}
Let $T$ be an $\IR$-tree and $S\subset T$ be 
a closed $\IR$-subtree.
Then the projection $p\colon T\to S$ (see 2.1.9 in \cite{chiswell}) 
extends to the quotient map $\hat{p}\colon \hat{T} \to \hat{S}$. 
In other words, 
$p_*\colon \cC(S) \ni h\mapsto h\circ p \in \cC(T)$ 
is an isometric embedding that 
sends $h_{u,v}$, $u,v\in S$, to $h_{u,v}$. 
\end{lem}
\begin{proof}
For every $u,v\in S$ and $x\in T$, one has $h_{u,v}(x) = h_{u,v}(p(x))$. 
\end{proof}

\begin{rem}\label{rem:sepRtree}
Let $T$ be an $\IR$-tree. 
\begin{enumerate}
\item 
The topology of $\hat{T} = T\cup\partial T$ 
is generated by the \emph{half spaces} 
\[
H_u(v) \coloneq \{ x\in \bar{T} : \ip{x,v}_u > 0 \} = \{ x\in\bar{T} : u\notin[v,x]\} 
\] 
with $u,v\in T$. 
\item 
Every isometric isomorphism on $T$ extends to 
a homeomorphism on $\hat{T}$. 
\item
A point $x$ in $T$ is called a \emph{leaf} 
if $T\setminus\{x\}$ is connected, or equivalently if $x$ is not contained 
in any open arc. 
We will ignore the trivial case where $T$ consists of a single point. 
We denote the set of leafs by $L(T)$ and call the $\IR$-subtree 
$T^\circ \coloneq T\setminus L(T)$ a \emph{skeleton}.  
We observe that replacing $T$ with $T^\circ$ 
has no effect on the tree compactification, but 
$\partial (T^\circ) = \partial T \sqcup L(T)$. 
\item\label{remitem:sep}
Suppose that $T$ is separable (in the original metric) 
and $\{y_k\}$ is a countable dense subset of $T$.
Then 
the tree compactification $\hat{T}$ is second countable as 
$\cC(T)$ is generated by the countable subset $\{h_{y_k,y_l} : k\neq l\}$. 
Moreover, $T$ has at most countably many branch points, 
since it is enough to consider the tripods spanned by triples from $\{ y_k\}$. 
For every $z\in T$, the function 
$x\mapsto\dist(x,z) = \sup_k\ip{x,y_k}_z$ is Borel on $T \subset \hat{T}$. 
Thus the skeleton $T^\circ = \bigcup_{k,l} (y_k,y_l)$ is a Borel subset 
of $\hat{T}$ and the Borel structures on $T^\circ$ 
induced by $\hat{T}$ and by the original metric coincide.
\end{enumerate}
\end{rem}

\begin{lem}\label{lem:nbst}
Let $T$ be a non-trivial $\IR$-tree on which 
a group $\Ga$ acts isometrically. 
Then for every $x\in \hat{T}$ that is not a branch point, 
the stabilizer subgroup $\Ga^x$ at $x$ is an abelian extension 
of a directed union of arc stabilizers. 
In particular, it is amenable if all arc stabilizers are amenable. 
\end{lem}
\begin{proof}
We may assume $T=T^\circ$. 
Let $x\in\partial T$ be the endpoint of a ray $\alpha = [o,x)$  
then $s\mapsto \lim_{y\in\alpha} \dist(sy,o)-\dist(y,o)$ defines 
a homomorphism from $\Ga^x$ into $\IR$ whose kernel is 
the union of the stabilizers of arcs $[y,x)$ with $y\in\alpha$. 
Next, suppose that $x\in T$ is not a branch point 
and let $\Delta\le\Ga^x$ be the union 
of the stabilizers of open arcs containing $x$. 
Since every $s\in \Ga^x$ either fixes an open arc 
containing $x$ or flips around $x$, 
the subgroup $\Delta$ has index at most two in $\Ga^x$. 
\end{proof}
\section{$\La$-trees and the tree compactification}
Let $\La$ be an ordered abelian (additive) group. 
Every subgroup inherits the order. 
A subgroup $M \le \La$ is \emph{convex} if 
for every $b\in\La$, that $|b|<a$ for some $a\in M_+$ implies $b\in M$. 
We denote by $M^\circ$ the union of proper convex subgroups of $M$.
For every $a\in\La_+$, 
the convex subgroup $\ip{a}_{\conv}$ generated by $a$ is given by 
\[
\ip{a}_{\conv} \coloneq \{ b\in\La : \exists n \mbox{ such that } |b| < na\}
\]
and its largest proper convex subgroup $\ip{a}_{\conv}^\circ$ is given by 
\[
\ip{a}_{\conv}^\circ \coloneq \{ b\in\La : \forall n\quad n|b| < a\} = \ker \ve^a.
\]
Here $\ve^a\colon \ip{a}_{\conv} \to\IR$ is the unique ordered homomorphism 
that sends $a$ to $1$ (Theorem 1.1.2 in \cite{chiswell}). 
For $a,b\in\La_+$, 
we write $b\asymp a$, $b\preceq a$, and $b \ll a$ 
respectively if 
$\ip{b}_{\conv}=\ip{a}_{\conv}$, $b\in \ip{a}_{\conv}$, 
and $b \in \ip{a}_{\conv}^\circ$. 
Note that if $b\asymp a$, 
then $\ve^b = \ve^a(b)^{-1}\ve^a$; 
Also that if $\La \le \IR^\cU$ 
(see Section~\ref{sec:tame}) and $a = [a_n]_n\in\La$, then 
$\ve^a( b ) = \Lim_\cU b_n/a_n$ for $b = [b_n]_n \in \ip{a}_{\conv}$. 

A $\La$-metric space is a space $X$ equipped with 
a $\La$-metric $\dist$. 
The Gromov product $\ip{x,y}_u\in\frac{1}{2}\La$ is defined verbatim. 
A $\La$-metric space $X$ is a \emph{$\La$-tree} 
if it is geodesic, $0$-hyper\-bolic, and $\ip{x,y}_u\in\La$. 
See \cite{chiswell} for a comprehensive treatment 
of $\La$-trees (in particular, 2.1.6 and 2.4.3 for 
the above definition of a $\La$-tree). 
When $\La=\IZ$, a $\La$-tree is nothing but a simplicial tree. 
When $\La=\IR$, the notions of $\La$-tree and $\IR$-tree coincide. 
For a $\La$-tree $X$ and $u,v\in X$, we define 
the function $h_{u,v}$ on $X$ by $h_{u,v}(x)=\ve^{\dist(u,v)}(\ip{x,v}_u)$. 
We denote by $\cC(X)\subset\ell_\infty X$ 
the $\rC^*$-subalgebra generated by $\{ h_{u,v} : u,v\in X\}$. 
We define the \emph{tree compactification} $\hat{X}$ to be 
the Gelfand spectrum of the unital commutative \cst{} $\cC(X)$. 
This definition is compatible with the one for 
a $\IZ$-tree (see Section 5.2 in \cite{bo}) 
and for an $\IR$-tree. 
Every isometry on $X$ extends to a homeomorphism on $\hat{X}$.

For a $\La$-tree $X$, $x\in X$, and a convex subgroup $M\le\La$, 
we define the \emph{full $M$-subtree} centered at $x$ by 
\[
X(x,M) \coloneq \{ y\in X : \dist(y,x) \in M\}.
\]
For an ordered homomorphism $\ve\colon \La\to\IR$, 
we consider the associated $\IR$-valued pseudo-metric 
$\dist^\ve \coloneq\ve\circ\dist$ and denote by 
$\ve_*$ the quotient map from $X$ onto 
the $\ve(\La)$-tree $\ve_*(X) \coloneq X/{\sim}$, 
where $x\sim y$ if $\dist^\ve(x,y)=0$. 
This process is known as \emph{base change}. 

\begin{rem}\label{rem:Qtree}
Suppose that $T$ is a $\La$-tree with $\La\subset\IR$. 
There are two cases: either $\La$ is discrete or $\La$ is dense. 
The first case is when $T$ is a simplicial tree. 
The second case is handled by Theorem 2.4.7 in \cite{chiswell}: $T$ 
is completed to an $\IR$-tree $T'$ and 
$\cC(T') = \cC(T)$, or equivalently, $\hat{T}' = \hat{T}$.  
Note that every branch point of $T'$ comes from $T$ 
because every tripod in $T$ has its center in $T$ 
(Lemma 2.1.2 in \cite{chiswell}).
\end{rem}

\begin{lem}\label{lem:treeemb2}
Let $X$ be a $\La$-tree. 
For every $x\in X$ and every convex subgroup $M\le\La$, 
the canonical embedding
\[
\cC(X(x,M)) \hookrightarrow \cC(X),\ 
h_{u,v}\mapsto h_{u,v}
\]
is isometric. 
For every $x\in X$ and $a\in\La_+$, the canonical embedding
\[
\cC(\ve^a_* (X(x,\ip{a}_{\conv}))) \hookrightarrow \cC(X(x,\ip{a}_{\conv})),\ 
h_{\ve^a_*(u),\ve^a_*(v)}\mapsto h_{u,v}
\]
is isometric. 
\end{lem}
\begin{proof}
Let finitely many $u_i,v_i\in X(x,M)$ be given. 
Then the subtree $Y$ spanned by them is contained 
in $X(x,M)$. 
Hence for every $y \in X$, its projection $\pi(y)$ to $Y$ 
(see Section 2.1 in \cite{chiswell})
satisfies $h_{u_i,v_i}(\pi(y)) = h_{u_i,v_i}(y)$.
It follows that the ``identity'' map $h_{u,v}\mapsto h_{u,v}$ 
extends to an isometry from $\cC(X(x,M))$ 
into $\cC(X)$. 

For any $u,v\in X(x,\ip{a}_{\conv})$ with $\dist^{\ve^a}(u,v)\neq0$, 
one has $\dist(u,v)\asymp a$ and 
\[
h_{u,v}(y) = \ve^{\dist(u,v)} (\ip{y,v}_u) 
 = \ve^{a}(\dist(u,v))^{-1} \ve^{a} (\ip{y,v}_u) 
 = h_{\ve^a_*(u),\ve^a_*(v)}(\ve^a_*(y)).
\] 
This means that the map $h_{\ve^a_*(u),\ve^a_*(v)}\mapsto h_{u,v}$ 
extends to an isometry. 
\end{proof}

For a $\La$-tree $X$ and a finite set $\cS$ of 
isometric isomorphisms on $X$, 
we define the $\cS$-dis\-place\-ment at $x\in X$ by 
\[
\delta_{\cS}(x) \coloneq \max_{s\in\cS}\dist(sx,x) \in \La.
\]
We observe that $\dist(tx,x) \preceq \delta_{\cS}(x)$ 
for every $t$ in the group generated by $\cS$. 

\begin{lem}\label{lem:center}
For $X$ and $\cS$ as above, 
if $\cS$ does not contain inversions, 
then there is $x\in X$ that minimizes
$\ip{\delta_{\cS}(x)}_{\conv}$.
\end{lem}
\begin{proof}
Let $\ell(s)$ denote the translation length and 
$A_s$ denote the characteristic set of $s\in\cS$ 
(see Section 3.1 in \cite{chiswell}). 
One has $\dist(sy,y)=\ell(s)+2\dist(y,A_s)$ for every $y\in X$ 
by 3.1.1 and 3.1.4 in \cite{chiswell}. 
Let $a\coloneq \max_{s,s'\in\cS}\dist(A_s,A_{s'})$. 
If $a=0$, pick $x \in \bigcap A_s$. 
Otherwise, 
pick $g,g'\in\cS$ such that $\dist(A_g,A_{g'}) = a$ 
and let $x\in A_g$ be the projection of $A_{g'}$ to $A_g$. 
Note that for every $s\in\cS$ and $y\in X$ one has 
\[ \dist(x,A_s)\le a \le \dist(y,A_g)+\dist(y,A_{g'})
\]
by Lemma 2.1.10 in \cite{chiswell}. 
Let $b \coloneq \max_{s\in\cS} \ell(s)$. 
If $a \preceq b$, 
then $\delta_{\cS}(x) \preceq b$.
If $b \preceq a$, 
then $\delta_{\cS}(x) \preceq a$. 
In either case, $\delta_{\cS}(x) \preceq \delta_{\cS}(y)$.
\end{proof}
\section{Amenability of real trees}
\begin{thm}\label{thm:Rtreeequivrel}
Let $\La$ be either $\IZ$ or $\IR$ and 
let $\Ga\acts T$ be an isometric action of 
a countable group $\Ga$ on a separable $\La$-tree.
Let $B\subset T$ denote the countable set of branch points 
and $\hat{T}$ denote the tree compactification of $T$.
Assume that $T$ satisfies conditions 
$(\ref{cond:LaTreeCA})$ and $(\ref{cond:LaTreeAS})$ 
in Theorem~$\ref{thmA:LaTree}$.
Then the Borel $\Ga$-space $\hat{T}\setminus B$ is amenable. 
\end{thm}

Before entering the proof, we recall the notion of 
the \emph{length measure} $\mu$ on an $\IR$-tree $T$. 
See Section 4.3.5 in \cite{evans}. 
In case $T$ is separable, 
$\mu$ is the unique $\sigma$-finite atomless 
Borel measure whose restriction to an arc 
coincides with the associated Lebesgue measure. 
Every isometry on $T$ preserves the length measure.

\begin{proof}
By condition (\ref{cond:LaTreeAS}), all arc stabilizers 
for $\Ga\acts T$ are amenable and hence the stabilizer 
subgroup $\Ga^x$ at every $x\in \hat{T}\setminus B$ 
is amenable by Lemma~\ref{lem:nbst}. 
We assume $T=T^\circ$ and 
prove amenability of $T\setminus B$ and $\partial T$ separately. 

We prove amenability of $\Ga\acts T\setminus B$. 
Our proof is inspired by \cite{glp}.
Thanks to Lemma~\ref{lem:borelamencriterion}, 
it suffices to show that the orbit equivalence 
relation $\cR_{\Ga\acts T}$ is $\cM$-amenable. 
The case of simplicial tree is trivial as $T$ is countable. 
We assume that $T$ is an $\IR$-tree and take 
a countable dense subset $\{ y_k \}$ of $T$. 
Thus $T = T^\circ = \bigcup_{k,l}(y_k,y_l)$. 
By Lemma~\ref{lem:becd}, it suffices to show that, 
for every $n\in\IN$ and every finite subset $F\Subset\Ga$,
the countable Borel equivalence relation $\cS$ 
on $Y \coloneq \bigcup_{k,l \le n}(y_k,y_l)$ 
induced by $F$ is $\cM$-amenable. 
We write $\vp_t$ for the partial isometry that is $t$ restricted to 
$\dom\vp_t = \{ x\in Y : tx\in Y\}$. 
We observe that $Y$ has finitely many branch points and 
is decomposed into a finite family $\{ I_k\}$ of intervals 
(possibly degenerate) in such a way that each $\vp_t$, $t\in F$, 
permutes the set $\{ I_k\}$. 
Align $\{I_k\}$ in a straight line. 
Then $\vp_{t,k} \coloneq \vp_t|_{I_k}$ becomes a partial translation possibly after a flip. 
Hence $\cS$ is a subequivalence relation 
of the orbit equivalence relation of 
a countable subgroup of $\IR\rtimes(\IZ/2)$ acting on $\IR$.
Since $\IR\rtimes(\IZ/2)$ is amenable, 
the countable Borel equivalence relation $\cS$ is $\cM$-amenable. 

We prove amenability of  $\Ga\acts \partial T$. 
We only deal with the case $\La=\IR$, 
as the other case $\La=\IZ$ is 
simpler (and well-known as well) 
by repeating the following proof 
for the edge paths. 
Let $\zeta_i\colon T\setminus B \to \Prob \Ga$ 
be an approximate equivariant mean, 
whose existence is already proved above. 
Let $\mu$ denote the length measure on $T$ 
and fix a base point $o\in T$. 
We extend the metric $\dist$ to $\dist\colon (T\cup\partial T)\times T\to[0,\infty]$ 
in the obvious way and set 
$\partial_\fin T\coloneq\{ x\in\partial T :  \dist(x,o) < \infty\}$ 
and $\partial_\infty T\coloneq \partial T\setminus\partial_\fin T$.
These subsets are Borel in $\hat{T}$. 
We observe that 
\[
\Omega_n \coloneq \{ (x,z) \in \partial_\fin T \times T 
 : z\in[o,x),\,\dist(x,z)\le 1/n\}
\] 
is a Borel subset of $\partial_\fin T \times T$, 
since $\hat{T} \times T \ni (x,z) \mapsto h_{o,z}(x)$ 
is continuous as $T$ in the second variable 
is equipped with the original metric topology; 
see Remark~\ref{rem:sepRtree}.(\ref{remitem:sep}). 
We denote by $\upsilon_n$ the characteristic 
function for $\Omega_n$. 
Observe that $\upsilon_n^x \in L^1(T,\mu)_+$ with 
$\|\upsilon_n^x\|_1 = (1/n)\wedge\dist(x,o)$. 
Thus $\eta_n(x,z) \coloneq \|\upsilon_n^x\|_1^{-1}\upsilon_n(x,z)$ 
satisfies $\|\eta_n^x\|_1=1$ 
and 
\[
\| \eta_n^{sx} - s\eta_n^x \|_1 = 0
\]
for every $s\in\Ga$ and $x\in\partial_\fin T$ 
as soon as $n$ is large enough so that 
the boundary point $sx$ is at least $1/n$ 
away from the arc $[o,so]$. 
Similarly, we define the Borel function 
$\eta_n$ on $\partial_\infty T \times T$ to be 
$1/n$ times the characteristic function for 
$\{ (x,z) : z\in[o,x),\,\dist(z,o)\le n\}$. 
One has $\|\eta_n^x\|_1=1$ and
\[
\| \eta_n^{sx} - s\eta_n^x \|_1 \le \frac{2}{n}\dist(o,so)
\]
for every $s\in\Ga$ and $x\in\partial_\infty T$ (see Lemma 5.2.6 in \cite{bo}).
With $\mu(B)=0$ in mind, we define the Borel map 
$\xi_{n,i}\colon \partial T\to\Prob \Ga$ by 
$\xi_{n,i}^x (t) \coloneq \int \eta_n(x,z)\zeta_i^z(t)\,d\mu(z)$.
It satisfies 
\[
\| \xi_{n,i}^{sx} - s\xi_{n,i}^x \|_1
 \le \| \eta_n^{sx} - s\eta_n^x \|_1 + \int \eta_n^x(z)\|\zeta_i^{sz} - s\zeta_i^z\|_1\,d\mu(z).
\]
Let a finite subset $F\Subset\Ga$, $m\in\Prob(\partial T)$, and $\ve>0$ be given.
There is $n$ such that 
$\sum_{s\in F}\int \| \eta_n^{sx} - s\eta_n^x \|_1 \,dm(x)\le \ve/2$. 
Since $z\mapsto\int\eta_n^x(z)\,dm(x)\mu(z)$ is a probability 
measure on  $T\setminus B$, 
there is $i$ such that 
$\sum_s\int\int\eta_n^x(z)
\|\zeta_i^{sz} - s\zeta_i^z\|_1\,dm(x)\,d\mu(z)<\ve/2$.
For these $n$ and $i$, one has \\
\begin{minipage}{.9\textwidth}
\[
\sum_{s\in F} \int \| \xi_{n,i}^{sx} - s\xi_{n,i}^x \|_1\,dm(x)<\ve.
\]
\end{minipage}
\end{proof}
\section{Proof of Theorem~\ref{thmA:LaTree}}
\begin{defn}
Let $\Ga\acts X$ be an isometric action 
of a group $\Ga$ on a $\La$-metric space $X$.
For every $x\in X$, we define the equivalence 
relation $\asymp_x$ on $\Ga$ by $s\asymp_x t$ if 
$\dist(sx,x) \asymp \dist(tx,x)$ in $\La$. 
We define the \emph{$\La$-rank} of 
the action $\Ga\acts X$ by 
\[
\rank_\La(\Ga\acts X)
 \coloneq \sup\{ | \Ga/{\asymp_x} | -1 : x\in X\}
 \in \{0,1,\ldots,\infty\}.
\]
\end{defn}
\begin{proof}[Proof of Theorem~\ref{thmA:LaTree}]
We proceed by induction on $d \coloneq \rank_\La(\Ga\acts X)$. 
If $d=0$, then $\Ga$ acts trivially on $X$ and hence $\Ga=\Ga^x$ 
acts amenably on $K$. 
Suppose that Theorem is already proved up to $d-1$ and let's prove 
amenability of $\Ga\acts K$. 
By passing to a subgroup, we may assume that $\Ga$ is finitely generated, 
say by a subset $\cS\Subset\Ga$. 
Observe that this passage does not affect 
the conditions in Theorem. 
Let $x\in X$ be as in Lemma~\ref{lem:center} 
and set $a\coloneq\delta_{\cS}(x)$ and $M \coloneq \ip{a}_{\conv}$. 
Then $\Ga \acts X(x,M)$. 
Let $\La_0 \coloneq \ve_*^a(M) \subset\IR$ and 
$\ve_*^a(X(x,M))$ be 
the $\La_0$-tree obtained by base change. 
Take a separable $\Ga$-invariant $\La_0$-subtree 
$T_0\subset \ve_*^a(X(x,M))$. 
If $\La_0$ is discrete in $\IR$, then set $T \coloneq T_0$, which is a simplicial tree.
Otherwise, set $T$ to be the $\IR$-tree obtained 
by completion of $T_0$ 
(see Remark~\ref{rem:Qtree}). 
In either cases, $\Ga$ acts on $T$ isometrically and 
$\Ga\acts T$ satisfies the conditions 
$(\ref{cond:LaTreeCA})$ and $(\ref{cond:LaTreeAS})$, 
with $\La$ replaced by $\overline{\Lambda_0}$. 
By condition (\ref{cond:LaTreeAS}), all arc stabilizers for $\Ga\acts T$ are amenable. 
By Lemmas~\ref{lem:treeemb1} and \ref{lem:treeemb2}, 
one has isometric embeddings 
\[
\cC(T)\subset\cC( \ve_*^a(X(x,M))) \subset \cC(X)\subset C(K).
\]
In other words, there is a $\Ga$-equivariant quotient 
map $Q$ from $K$ onto $\hat{T}$. 

We are going to invoke Proposition~\ref{prop:amencriterion}. 
Let $B \subset T$ denote the countable 
subset of branch points. 
By Theorem~\ref{thm:Rtreeequivrel}, 
the Borel $\Ga$-space $\hat{T}\setminus B$ is amenable.
Let $b\in B$ be given and pick 
$y\in X(x,M)$ such that $\ve_*^a(y)=b$. 
Then one has 
$\Ga^b = \{ t\in\Ga : \dist(ty,y) \ll a \}$.
We consider the full $M^\circ$-subtree 
$Y \coloneq X(y,M^\circ)$ centered at $y$, 
on which $\Ga^b$ acts. 
By Lemma~\ref{lem:treeemb2}, $\hat{Y}$ is a quotient $\Ga^b$-space of $K$. 
We check that the action $\Ga^b\acts Y$ satisfies the conditions of Theorem 
with $\rank_{M^\circ}(\Ga^b\acts Y) \le d-1$. 
Only condition (\ref{cond:LaTreeFR}) needs a proof. 
Let $z\in Y$ be given arbitrary. 
By the choice of $x$, one has 
$a \preceq \delta_{\cS}(z)$. 
It follows that 
$\dist(tz,z) \ll \delta_{\cS}(z)$ for every $t\in\Ga^b$. 
This shows $|\Ga^b/{\asymp_z}| \le |\Ga/{\asymp_z}|-1\le d$. 
By the induction hypothesis, $K$ (and hence $Q^{-1}(b)$) 
is amenable as a compact $\Ga^b$-space. 
We are done by Proposition~\ref{prop:amencriterion}. 
\end{proof}
\section{Uniform exactness of free groups}\label{sec:tame}
We recall the ultraproduct construction 
for a $\La$-metric space $(Y,\dist)$. 
See Section~3 in \cite{guirardel} on this topic.
We set $Y^\cU \coloneq (\prod_{\IN} Y)/{\sim}$, where $(y_n)_n\sim(z_n)_n$ if 
$\{ n : y_n=z_n \}\in\cU$ and likewise for $\La^\cU$. 
Then $\La^\cU$ is an ordered abelian group 
(with the obvious order that 
$[a_n]_n\le[b_n]_n$ $\Leftrightarrow$ $\{ n : a_n\le b_n\}\in\cU$) 
and $Y^\cU$ is a $\La^\cU$-metric space with 
the $\La^\cU$-metric given by 
$\dist([x_n]_n,[y_n]_n) \coloneq [ \dist(x_n,y_n) ]_n$. 
Here we recall that $[x_n]_n$ etc is the element 
in the ultrapower that arises from the sequence $(x_n)_n$. 
We observe that the ultrapower $Y^\cU$ of a $\La$-tree $Y$ 
is a $\La^\cU$-tree. 
\begin{defn}
Let $G$ be a finitely generated group and denote by  
$\ell\colon G\to\IZ$ the word length function 
associated with a finite generating subset. 
We define the equivalence relation $\asymp$ 
on $G^\cU$ by $s\asymp t$ 
if $\ell^\cU(s) \asymp \ell^\cU(t)$ in $\IZ^\cU$. 
The equivalence relation $\asymp$ does not depend 
on the choice of a generating subset of $G$. 
We say that $G$ has \emph{tame geometry at infinity} 
if for every $d$ there is $r$ such that 
every $d$-generated subgroup $\Ga \le G^\cU$ 
has at most $r+1$ $\asymp$-equivalence classes, 
or equivalently $\rank_{\IZ^\cU}(\Ga\acts G^\cU)\le r$, 
where $G$ is viewed as the metric space 
equipped with the word distance 
$\dist(x,y) \coloneq \ell(x^{-1}y)$. 
\end{defn}

The notion of tameness is inspired by \cite{guirardel} 
that essentially proves 
the tameness of free groups in the middle of proving 
that every limit group acts freely on an $\IR^p$-tree. 
\begin{rem}
Here is an ultrafilter free characterization of tameness. 
A finitely generated group $G$ has tame geometry at infinity 
if and only if for every $d$ there is $r$ such that at most $r$ 
elements in the rank $d$ free group $\bF_d$ can have 
arbitrarily different magnitudes of lengths in $G$; 
namely for every 
$g_1,\ldots,g_{r+1}\in\bF_d$ there is $C>0$ such that 
every homomorphism $\pi\colon\bF_d\to G$ must have 
a pair $\{ i,j\}$ of distinct indices that satisfies
\[
\ell(\pi(g_i)) < C \ell(\pi(g_j))
\ \mbox{ and }\ 
\ell(\pi(g_j)) < C \ell(\pi(g_i))
\]
\end{rem}

Recall that a valuation on a field $K$ is a map $v$ 
from $K$ into $\La\cup\{\infty\}$, where $\La$ is 
an ordered abelian group called the value group, 
that satisfies $v(t)=\infty$ if and only if $t=0$, 
$v(st) = v(s)+v(t)$, and $v(s+t) \geq v(s)\wedge v(t)$ 
for all $s,t\in K$.  
Observe that the ultrapower 
$v^\cU\colon K^\cU\to\La^\cU\cup\{\infty\}$ 
is a valuation on the ultrapower field $K^\cU$. 
The map $v$ on $\IR$, given by $v(t) \coloneq -\log |t|$, 
is not a valuation, but is a ``quasi-valuation'' as 
it satisfies the last condition only up to 
a constant $\log 2$. 
Thus for the convex subgroup 
$\ip{1}_{\conv} \le \IR^\cU$, 
the induced map $[v^\cU]$ from the ultrapower 
field $\IR^\cU$ into the ordered abelian group 
$\IR^\cU/\ip{1}_{\conv}$ is a valuation, 
which is known as the \emph{natural valuation} of 
the ordered field $\IR^\cU$.

\begin{thm}[\cite{guirardel} for $K=\IQ_p$ 
and GPT-5.6 for $K=\IR$]\label{thm:tame}
Let $K\in\{\IQ_p,\IR\}$ and 
denote by $v\colon K\to \IR\cup\{\infty\}$ 
the $p$-adic valuation when $K=\IQ_p$ or 
the ``quasi-valuation'' as above when $K=\IR$. 
For $g \in \GL(m,K)$, set 
\[
\nu(g) \coloneq -\min\{ v(g_{i,j}), v((g^{-1})_{i,j}) : i,j\},
\]
which is equivalent to $\log(\|g\|\vee\|g^{-1}\|)$ when $K=\IR$.
Let $G\le \GL(m,K)$ be a finitely generated subgroup 
such that the word length $\ell$ is equivalent 
to $\nu$ on $G$:  
there are $C,D\geq0$ that satisfies 
\[
C^{-1}\nu(g) - D \le \ell(g) \le C\nu(g) + D
\]
for all $g\in G$. 
Then $G$ has tame geometry at infinity.
\end{thm}
\begin{proof}
We naturally identify $\GL(m,K)^\cU$ 
with $\GL(m,K^\cU)$. 
Accordingly, the ultrapower length $\ell^\cU$ 
on $G^\cU$ is equivalent to $\nu^\cU$, 
which is given by the same 
formula as $\nu$, but replacing $v$ with $v^\cU$. 
We denote the diagonal copy of $K$ 
in the ultrapower field $K^\cU$ by $L_0$. 
Let's recall from the preceding discussion 
the natural valuation $[v^\cU]$ 
on $K^\cU$ for the case $K=\IR$. 
For notational consistency, 
we also consider for the case $K=\IQ_p$ the valuation 
$[v^\cU]\colon K^\cU\to\IZ^\cU/\ip{1}_{\conv}$ 
induced by $v^\cU$. 
Accordingly, the corresponding map on $\GL(m,K^\cU)$ 
is denoted by $[\nu^\cU]$. 
Let a finitely generated subgroup 
$\Ga =\ip{\cS}\le G^\cU$ be given. 
Then for the subfield $L\le K^\cU$ 
that is generated by $L_0$ and the 
$m^2|\cS|$ entries of elements in $\cS$, 
the map $[\nu^\cU]$ on $\Ga$ 
takes values in the value group $[v^\cU](L^{\times})$. 
Note that $[v^\cU](L_0^{\times})=\{0\}$. 
By Theorem 3.5 in \cite{guirardel}, the ordered abelian 
group $[v^\cU](L^{\times})$ has rank 
at most $m^2|\cS|$. 
This implies that there are at most $m^2|\cS|-1$
equivalence classes 
of $\ell^\cU$ values on the ``unbounded'' 
subset $\Ga\setminus G$ of $\Ga$. 
Since $\ell^\cU=\ell$ on $G$ has 
exactly two equivalence classes, 
it satisfies the definition of tameness with $r = 3d+1$. 
\end{proof}

\begin{cor}\label{cor:tame}
Finitely generated free groups have tame geometry at infinity 
and so do the fundamental groups 
of closed real hyperbolic manifolds.
\end{cor}
\begin{proof}
A finitely generated free group $\bF$ is a 
uniform lattice in $\SL(2,\IQ_p)$ and 
acts freely and cocompactly on the associated 
Bruhat--Tits tree.  
By the \v{S}varc--Milnor lemma, 
$\ell$ is equivalent to $g\mapsto\dist(gx_0,x_0) = 2\nu(g)$ 
for the base point $x_0$ (see Lemma 4.3.5 in \cite{chiswell}). 
Similarly, the fundamental group $G$ of a closed 
real hyperbolic manifolds is a uniform lattice in 
$O^+(n,1)\le \GL(n+1,\IR)$ and acts freely and 
cocompactly on the real hyperbolic space $\IH^n$. 
By the \v{S}varc--Milnor lemma, 
$\ell$ is equivalent to 
$g\mapsto\dist(gx_0,x_0)=\log\|g\|$ 
for the base point $x_0=(0,\ldots,0,1)$, 
where $\|g\| = \| g^{-1}\|$ is the operator norm 
of $g \in \GL(n+1,\IR)$ and $\log\|g\|$ is 
equivalent to $\nu(g)$ on $O^+(n,1)$. 
Thus Theorem~\ref{thm:tame} applies in both cases. 
\end{proof}
\begin{thm}[Theorem \ref{thmA:freegroup}]\label{thm:freegrp}
Free groups $\bF$ are uniformly exact. 
\end{thm}
\begin{proof}
We write $Y$ for the free group 
$\bF$ viewed as a $\IZ$-tree. 
For amenability of $\bF^\cU\acts(\ell_\infty Y)^\cU$, 
it suffices to show that every finitely generated subgroup $\Ga \le \bF^\cU$ 
acts amenably on $(\ell_\infty Y)^\cU$.
We recall the canonical embedding 
$(\ell_\infty Y)^\cU\subset \ell_\infty(Y^\cU)$, 
given by $[f_n]_n([x_n]_n) = \Lim_\cU f_n(x_n)$. 
We observe that $\cC(Y^\cU)\subset (\ell_\infty Y)^\cU$, since 
$h_{u,v} = [h_{u_n,v_n}]_n$ belongs to $(\ell_\infty Y)^\cU$ for 
every $u=[u_n]_n$ and $v=[v_n]_n$ in $Y^\cU$. 
We will invoke Theorem~\ref{thmA:LaTree} to prove 
amenability of $\Ga\acts\cC(Y^\cU)$. 

We have to check the conditions of Theorem~\ref{thmA:LaTree} 
for the tree compactification of $Y^\cU$. 
The set of amenable subgroups in the finitely generated limit group $\Ga$ is 
countable, because every subgroup $\Delta$ of $\Ga$ that does not 
contains a non-cyclic free group is abelian and finitely generated 
(see Corollary 4.4 in \cite{sela} or Theorem~6 in \cite{km}; 
and \cite{sela:plms,km2} for a limit of 
a torsion-free hyperbolic group $G$).
The action $\Ga\acts Y^\cU$ is free as $\bF\acts Y$ is free. 
Condition (\ref{cond:LaTreeAS}) follows from Lemma 1.3 in \cite{sela}, 
but we sketch the proof here for the readers' convenience. 
Let $x,y\in Y^\cU$, $x\neq y$, be given and let $\Delta$ be the 
``rough arc stabilizer'' subgroup presented in 
Theorem~\ref{thmA:LaTree}.(\ref{cond:LaTreeAS}). 
We claim that the set of commutators in $\Delta$ is finite, 
which implies that $\Delta$ cannot contain a non-cyclic free group. 
Let's write $x=[x_n]_n$ and $y=[y_n]_n$ and take $s=[s_n]_n$ and $t=[t_n]_n$ in $\Delta$. 
We may assume that $\dist(x_n,y_n)\to\infty$. 
Let $z_n$ be a midpoint (roughly) of $[x_n,y_n]$. 
Since $Y$ is a tree and 
$\dist(s_nx_n,x_n)+\dist(s_ny_n,y_n) \ll \dist(x_n,y_n)$ 
in the colloquial sense of $\ll$, the isometry $s_n$ roughly acts 
as a translation by $\dist(s_nx_n,x_n)$ on the midway of $[x_n,y_n]$. 
The same holds for $t_n$ and hence $[s_n,t_n] z_n$ is contained 
in a ball centered at $z_n$ with radius independent of $n$ 
(or more precisely, for $\cU$-generic $n$). 
Since the $\bF$-action on $Y$ is free, the claim is proved. 
Finally, condition (\ref{cond:LaTreeFR}) follows from 
Corollary~\ref{cor:tame}.
\end{proof}

\begin{rem}
We note that every torsion-free hyperbolic group $G$ 
having tame geometry at infinity is also uniformly exact. 
This can be proved in the same way as with free groups 
by working with the $(\IZ^\cU/\IZ)$-tree $G^\cU/G$ and 
$C(K) \coloneq (\ell_\infty G)^\cU$ on which every 
conjugate of $G$ inside $G^\cU$ acts amenably. 
The author has a sentiment that torsion-free hyperbolic 
groups ought to have tame geometry at infinity.
\end{rem}
\section{Modulus of uniform exactness}\label{sec:mod}
All results in this section are derived 
based on the traditional recipe 
on exactness (see, e.g., \cite{bo}), but care must be taken not 
to lose the uniform estimate. 
We will use theory of operator spaces and completely bounded maps. 
See Appendix~B in \cite{bo} and \cite{pisier} for 
a comprehensive treatment on this topic. 
\begin{defn}\label{defn:unifex}
Recall that a group $\Ga$ is exact 
if the action $\Ga\acts\ell_\infty\Ga$ is amenable, or equivalently, 
for every finite unital symmetric subset $F\Subset\Ga$ 
and $\ve>0$, there are a finite subset $E\Subset\Ga$ and 
a map $\eta\colon\Ga\to\Prob \Ga$ 
such that 
\[
\bigcup_{x\in\Ga}\supp\eta^x\subset E
\ \mbox{ and }\ 
\sum_{s\in F}\sup_{x\in\Ga}\|\eta^{sx} - s\eta^x\|_1 \le \ve. 
\]
See, e.g., Theorem 5.1.6 in \cite{bo}. 
Here we say $F$ is \emph{unital} (resp.\ \emph{symmetric}) 
if $1\in F$ (resp.\ $t\in F$ implies $t^{-1}\in F$).
Let $k\colon \IN \to \IN$ be given. 
We say a group $\Ga$ has \emph{modulus of 
uniform exactness} $k$ if the above $E$ (and $\eta$) 
for $F$ and $\ve = |F|^{-1}$ can be taken as $E = F^{k(|F|)}$. 
\end{defn}
We observe that the class $\UE(k)$ 
of groups having $k$ as modulus of uniform exactness 
is closed under passing to subgroups and directed unions.
If $\Delta \in \UE(k)$ is amenable, then it is uniformly amenable, 
because for every finite unital symmetric subset $F\Subset\Delta$, 
the integration $\int\eta^x\,dm(x)$ w.r.t.\ an invariant 
mean $m$ on $\ell_\infty\Delta$ is an $(F,|F|^{-1})$-invariant 
mean in $\Prob\Delta$ that is supported on $F^{k(|F|)}$. 
See \cite{bozejko, keller, ks, wysoczanski}. 
More generally, for an uniformly exact group $\Ga$ and 
a normal amenable subgroup $\Delta \triangleleft \Ga$, 
one can pass $\eta$ for $\Ga$ to that for $\Ga/\Delta$. 

\begin{lem}
The class $\UE(k)$ is compact for every $k\in\IN^\IN$.
\end{lem}
\begin{proof}
Let $(\vp_n\colon\bF_d\to\Ga_n)_n$ be a convergent sequence 
in $\cM_d$. 
Let a finite subset $F_\infty\Subset \Ga_\infty$ be given 
and  lift it to $F\Subset\bF_d$ and set $F_n \coloneq \vp_n(F)$. 
One has $|F_n| = |F|$ eventually.  
If $\Ga_n\in\UE(k)$ for every $n$, then there is 
$\eta_n\colon \Ga_n\to \Prob \Ga_n$ that 
satisfies the relevant formula in Definition~\ref{defn:unifex}. 
We take a lift $\tilde{\eta}_n\colon \Ga_n\to \Prob {F^{k(|F|)}}$ 
through the eventually one-to-one map 
$\vp_n\colon F^{k(|F|)} \to F_n^{k(|F|)}$ 
and take a limit point $\tilde{\eta}_\infty$ 
of $\tilde{\eta}_n\circ\vp_n\colon\bF_d\to \Prob {F^{k(|F|)}}$, 
in the topology of $[0,1]^{\bF_d \times F^{k(|F|)}}$.
One has $\|\tilde{\eta}_\infty^x\|_1 = 1$ for all $x$, 
since $F^{k(|F|)}$ is finite. 
Hence there is $\eta_\infty\colon \Ga_\infty \to \Prob F_\infty^{k(|F|)}$ 
that satisfies $\eta_\infty\circ\vp_\infty = (\vp_\infty)_*\circ\tilde{\eta}_\infty$, 
where $ (\vp_\infty)_*\colon \Prob \bF_d \to \Prob \Ga_\infty$ is the push-out. 
Now, it is easy to check that $\eta_\infty$ has the required property.
\end{proof}

For a finite subset $E\Subset\Ga$, 
let $P_E$ denote the orthogonal projection 
from $\ell_2\Ga$ onto $\ell_2E$ and 
$\Phi_E\colon \IB(\ell_2\Ga)\to \IB(\ell_2E)$ denote 
the compression map given by $S\mapsto P_ESP_E$. 

\begin{lem}\label{lem:osex}
Consider the following conditions 
for finite unital symmetric subsets $F,E\Subset\Ga$ 
and $\ve>0$. 
\begin{enumerate}
\item\label{cond:cb1}
There is $\eta\colon\Ga\to\Prob \Ga$ that satisfies 
the displayed condition in Definition~\ref{defn:unifex}.
\item\label{cond:cb2}
For every \cst{} $A$ and 
$S=\sum_{s\in F} a_s\otimes \lambda(s) \in A\otimes\rC^*_\rr\Ga$, 
one has 
\[
(1-\ve)\|S\|_{A\otimes\rC^*_\rr\Ga} 
 \le \|(1\otimes P_E)S(1\otimes P_E)\|_{A\otimes\IB(\ell_2E)}
 \le \|S\|_{A\otimes\rC^*_\rr\Ga}.
\]
Equivalently, for $\cF\coloneq\lspan\lambda(F)\subset\rC^*_\rr\Ga$, 
one has $\|(\Phi_E|_\cF)^{-1}\|_{\cb}\le (1-\ve)^{-1}$. 
\end{enumerate}
Then $(\ref{cond:cb1})\Rightarrow(\ref{cond:cb2})$ holds. 
Conversely, if $(\ref{cond:cb2})$ is true and $|F|\ve<0.01$, 
then $(\ref{cond:cb1})$ holds with $E$ and $\ve$ 
exchanged by $E^k$ and $100(|F|\ve)^{1/2}$. 
Here $k = k(|E|,\ve)$ is a positive integer that depends 
on $|E|$ and $\ve$.
\end{lem}
\begin{proof}
\textbf{Ad}$(\ref{cond:cb1})\Rightarrow(\ref{cond:cb2})$.
We set $\xi\coloneq\eta^{1/2}$ and 
define $V_t\colon\ell_2\Ga \to \ell_2E$ 
by $V_t\delta_x \coloneq \xi^x(xt^{-1})\delta_{xt^{-1}}$ 
and $\Psi\colon\IB(\ell_2E)\to\IB(\ell_2\Ga)$ by 
$\Psi(S)=\sum_t V_t^*SV_t$. 
Then $\Psi$ is unital and completely positive.
A routine calculation shows that 
$(\Psi\circ\Phi_E)(\lambda(s)) = \lambda(s)f_s$, 
where $f_s(x) \coloneq \ip{ s\xi^x, \xi^{sx}}$ 
acts as a diagonal operator on $\ell_2\Ga$. 
Since 
\[
2\|1-f_s\|_\infty = \sup_x\| \xi^{sx} - s\xi^x \|_2^2
 \le \sup_x \| \eta^{sx} - s\eta^x \|_1,
\]
one has 
\[
\|\id_A\otimes (\Psi\circ\Phi_E|_\cF)
  - \id_{A\otimes\cF}\|
 \le\sum_{s\in F}\| 1 - f_s \|_\infty \le \ve.
\]

\textbf{Ad}$(\ref{cond:cb2})\Rightarrow(\ref{cond:cb1})$. 
We follows the proof of Theorem 5.1.6 in \cite{bo}. 
By Corollary B.11 in \cite{bo} 
and Arveson's extension theorem, 
there is a unital completely positive map 
$\psi\colon\IB(\ell_2 E)\to\IB(\ell_2\Ga)$ 
such that $\theta\coloneq \psi\circ\Phi_E$ satisfies 
$\|\theta|_{\cF} - \id_{\cF} \|_{\cb}\le 4|F|\ve$.
We define the positive definite kernel 
$u$ on $\Ga$ by 
$u(x,y)=\ip{\theta(\lambda(xy^{-1}))\delta_y,\delta_x}$. 
One has $\supp u\subset\{(x,y) : xy^{-1}\in E^2\}$ and 
$| 1 - u(sx,x) | \le 4|F|\ve$ for $s\in F$. 
Hence the positive operator $a_u \in \IB(\ell_2\Ga)$ 
that corresponds to the kernel $u$ has norm 
at most $|E^2|$. 
By approximating $r\mapsto r^{1/2}$ on $[0,|E^2|]$ by 
a polynomial $p$, say of degree $k = k(|E|^2,\ve)$, 
one finds $b \coloneq p(a_u) \in \IB(\ell_2\Ga)_+$ 
such that $\|a_u - b^2 \|<\ve$ and that 
$\ip{b\delta_y,\delta_x}\neq0$ only if 
$xy^{-1} \in E^{2k}$.
Set $\eta^x(t) \coloneq (b\delta_x)(t^{-1}x)^2$. 
Then $\supp\eta^x\subset E^{2k}$,
$\|\eta^x\|_1=\|b\delta_x\|_2^2 \approx_\ve u(x,x)=1$, 
and, for every $s\in F$, 
\[
\|\eta^{sx} - s\eta^x\|_1
 = \| (b\delta_{sx})^2 - (b\delta_x)^2\|_1
 \le \| b\delta_{sx} - b\delta_x\|_2 \| b\delta_{sx} + b\delta_x\|_2
 \le 10(|F|\ve)^{1/2}.
\]
Hence $x\mapsto\eta^x/\|\eta^x\|_1$ satisfies condition (\ref{cond:cb1}) 
for the relevant constants. 
\end{proof}

The above is for Proposition~\ref{propA:norm}, 
and the below is for Theorem~\ref{thmA:unifamen} and Corollary~\ref{corA:list}.

\begin{thm}\label{thm;unfiexmod}
For a group $\Ga$, the following are equivalent.
\begin{enumerate}
\item\label{equiv:unifex} 
The group $\Ga$ is uniformly exact. 
\item\label{equiv:unifexmod}
The group $\Ga$ has modulus of uniform exactness. 
\item\label{equiv:unifexmod2}
There is $k\colon \IN\times(0,1) \to \IN$ with
the following property. 
For every finite subset $F\Subset\Ga$ and $\ve>0$, 
there are a finite subset $E\Subset\Ga$ with $|E|\le k(|F|,\ve)$ 
and a map $\eta\colon\Ga\to\Prob \Ga$ 
that satisfies the displayed condition in Definition~\ref{defn:unifex}.
\item\label{equiv:unifexcst}
For every \cst{} $A$, the canonical embedding 
\[
A^\cU \otimes \rC^*_\rr(\Ga^\cU)
 \hookrightarrow (A \otimes \rC^*_\rr\Ga)^\cU
\]
is continuous and isometric. 
\end{enumerate}
\end{thm}
\begin{proof}
\textbf{Ad}$(\ref{equiv:unifex})\Rightarrow(\ref{equiv:unifexmod})\Rightarrow(\ref{equiv:unifexmod2})\Rightarrow(\ref{equiv:unifex})$. 
We only prove the first implication as the rest are trivial. 
Observe that a compact $\Ga$-space $K$ is amenable 
if and only if for every finite unital symmetric subset $F\Subset\Ga$ 
and $\ve>0$ there are a finite subset $E\Subset\Ga$ and 
a family $(f^t)_{t\in \Ga}$ of non-negative functions 
(via $f^t(x) := \eta^x(t)$ for $\eta\colon K\to\Prob \Ga$ 
in Definition~\ref{def:topoamen}) such that 
\[
f_t=0
\mbox{ for $t\notin E$, }\ 
\sum_{t\in\Ga} f^t =1 
\ \mbox{ and }\ 
\sum_{s\in F}\sum_{t\in \Ga} |sf^t-f^{st}|<\ve
\] 
on $K$. 
We may assume by coset decomposition that 
$E$ is contained in the subgroup generated by $F$ 
and hence $E\subset F^k$ for some $k$.
By this correspondence (and modulo perturbation), 
amenability of $\Ga\acts(\ell_\infty\Ga)^\cU$ means the following: 
For every finite unital symmetric subset $F\Subset \Ga^\cU$ 
and $\ve>0$, there are $k\in\IN$ and 
$\eta_n\colon \Ga\to \Prob \Ga^\cU$ such that 
\[
\bigcup_{n\in\IN,\,x\in\Ga} \supp\eta_n^x \subset F^k
\ \mbox{ and }\ 
\Lim_\cU\sup_{x\in \Ga} \sum_{s\in F}\|\eta_n^{s(n) x} - s\eta_n^x\|_1 \le \ve, 
\]
where $s=[s(n)]_n$ for $s\in F$. 
Via lifting $F^{k+1}$ to $\prod_\IN\Ga$, this is interpreted as follows. 
For every family of finite subsets $F_n\Subset\Ga$ 
of size $d$ and every $\ve>0$, 
there are $k$ and $\eta_n\colon \Ga\to\Prob \Ga$ 
such that 
\[
\bigcup_{x\in\Ga} \supp\eta_n^x\subset F_n^k
\ \mbox{ and }\ 
\Lim_\cU \sup_{x\in \Ga}\sum_{s\in F_n}\|\eta_n^{s x} -s\eta_n^x\|_1 < \ve.
\]
This proves (\ref{equiv:unifexmod}) by contrapositive. 

\textbf{Ad}$(\ref{equiv:unifex})\Rightarrow(\ref{equiv:unifexcst}).$ 
Suppose $\Ga\in\UE$ and let $A\subset\IB(\cH)$ be given.  
Then the \cst{} $(A\otimes \rC^*_{\rr}\Ga)^\cU$ 
is faithfully represented on the ultrapower Hilbert space
$(\cH\otimes\ell_2\Ga)^\cU$. 
Since $(\ell_\infty\Ga)^\cU$ is nuclear, 
the representation of 
$A^\cU\otimes(\ell_\infty\Ga)^\cU$ on 
the same Hilbert space $(\cH\otimes\ell_2\Ga)^\cU$ 
is continuous. 
By Theorem~ 4.4.3 in \cite{bo} 
or \cite{anantharaman-delaroche,adr}, 
amenability of the action $\Ga^\cU\acts(\ell_\infty\Ga)^\cU$ 
implies continuity of the representation
$(A^\cU\otimes(\ell_\infty\Ga)^\cU)\rtimes_{\rr}\Ga^\cU 
\hookrightarrow \IB((\cH\otimes\ell_2\Ga)^\cU)$. 
As $A^\cU \otimes \rC^*_{\rr}(\Ga^\cU)$ is 
canonically contained in the former, 
this proves (\ref{equiv:unifexcst}). 

\textbf{Ad}$(\ref{equiv:unifexcst})\Rightarrow(\ref{equiv:unifexmod}).$ 
Suppose that $(\ref{equiv:unifexcst})$ holds, 
but $(\ref{equiv:unifexmod})$ does not. 
Because $(\ref{equiv:unifexmod})$ does not hold, 
by Lemma~\ref{lem:osex}, 
there are $d\in\IN$, $\ve>0$, and a sequence 
$F_n\Subset\Ga$ with $|F_n|=d$ such that 
for every $n$, $E_n \coloneq F_n^n$, and 
$\cF_n\coloneq\lspan\lambda(F_n)\subset \rC^*_\rr\Ga$, 
one has $\|(\Phi_{E_n}|_{\cF_n})^{-1}\|_{\cb}>1+\ve$. 
Set $\cG_n \coloneq \Phi_{E_n}(\cF_n) \subset \IB(\ell_2E_n)$. 
By operator space duality, the map 
$(\Phi_{E_n}|_{\cF_n})^{-1}\colon\cG_n\to\cF_n$ 
corresponds to an element $u_n\in \cG_n^*\otimes\cF_n$ 
with $\|u_n\| = \|(\Phi_{E_n}|_{\cF_n})^{-1}\|_{\cb}$. 
We write $F_\cU\subset\Ga^\cU$ for the subset of $d$ elements 
arising from $(F_n)_n$ and 
$\cF_\cU \coloneq \lspan\lambda(F_\cU) 
 \subset \rC^*_\rr(\Ga^\cU)$.
NB: Because of (\ref{equiv:unifexcst}), $\cF_\cU$ is actually 
completely isometrically isomorphic to the ultraproduct operator 
space of $(F_n)_n$. 
See Section 2.8 in \cite{pisier} for ultraproduct operator spaces. 
We write $\cG_\cU$ for the ultraproduct operator spaces 
of $(\cG_n)_n$. 
For the ultraproduct map
$\Phi_\cU\coloneq [\Phi_{E_n}|_{\cF_n}]_n$ from $\cF_\cU$ 
into $\cG_\cU$, the inverse $\Phi_\cU^{-1}$ corresponds 
to $u \in (\cG_\cU)^*\otimes \cF_\cU$ with 
$\|u\| = \| \Phi_\cU^{-1}\|_{\cb}$. 
We write $\tau$ for the canonical tracial state on $\rC^*_\rr(\Ga^\cU)$. 
Then for every $m$ and $S = [S_n]_n \in \IM_m\otimes\cF_\cU$, 
one has 
\begin{align*}
\| S \| &= \lim_k (\mathrm{tr}\otimes\tau)((S^*S)^k)^{1/2k}\\
 &= \lim_k \|(\id\otimes\Phi_{F_\cU^k})(S) \| 
  = \lim_k\Lim_\cU \|(\id\otimes\Phi_{F_n^k})(S_n) \| \\
 &\le \Lim_\cU \|(\id\otimes\Phi_{E_n})(S_n) \| 
  = \|(\id\otimes\Phi_\cU)(S) \|. 
\end{align*}
It follows that $\Phi_\cU$ is a complete isometry and 
$\|u\|=1$. 
With the canonical complete isometry 
$[\cG_n^*]_n = (\cG_\cU)^*$  
(Lemma 2.8.1 in \cite{pisier}) in mind, 
we consider a \cst{} $A$ containing 
copies of the operator spaces $\cG_n^*$. 
Then condition (\ref{equiv:unifexcst}) implies
\[
\Lim_\cU\| (\Phi_{E_n}|_{\cF_n})^{-1}\|_{\cb} 
 = \Lim_\cU \| u_n \|_{\cG_n^*\otimes\cF_n}
 = \| u \| = 1.
\]
This is in a contradiction with the hypothesis.
\end{proof}

It only remains to show Corollary~\ref{corA:list}.(\ref{cond:extension}). 
The permanence property for 
free products and graph products follows from that of extensions. 
Indeed, 
the group $\Ga_1 \coloneq \Ga*\IZ$ 
is an extension of the free group 
$\bigast_{\Ga_1/\IZ} \IZ$ by $\Ga$. 
The group $\Ga_1$ contains a copy of the infinite free product $\bigast_\IN\Ga$.
Also for a pair $\Delta\le\Ga$ and $\Ga'$, the group 
$\Ga^* \coloneq \Ga*_\Delta(\Delta\times\Ga')$ 
is an extension of 
$\bigast_{\Ga^*/(\Delta\times\Ga')} \Ga'$ by $\Ga$. 
Thus $\Ga^*$ is uniformly exact if $\Ga$ and $\Ga'$ are. 
The assertion for the graph products follows by induction 
on the size of graphs (see, e.g., Proof of Corollary 1.3 in \cite{gkmp}). 
See also \cite{lm} for related works. 

One way to prove the permanence property for extensions 
is to use the fact that $\Ga$ is uniformly exact if and only if 
for every \emph{twisted action} $(\alpha,\sigma)\colon\Ga\acts A$, 
the \emph{reduced twisted crossed product} satisfies 
\[
A^\cU \rtimes^{\sigma^\cU}_\rr \Ga^\cU
 \hookrightarrow (A \rtimes^\sigma_\rr \Ga)^\cU.
\]
This fact can be proved in the same way as 
Theorem \ref{thm;unfiexmod}.(\ref{equiv:unifexcst}).
For a normal subgroup $\Delta$ of $\Ga$, 
the reduced twisted crossed product 
$A\rtimes^\sigma_\rr \Ga$ has a decomposition 
into iterated reduced twisted crossed product by 
$\Delta$ and then by $\Ga/\Delta$, 
which is canonical after fixing a lift $\Ga/\Delta\to\Ga$. 
See \cite{bedos} for a detail. 
Instead, we prove here the permanence property 
more directly and \emph{effectively} by following the proof 
of Proposition 5.1.11 in \cite{bo}. 
\begin{proof}
Let $\Ga$ be a group with a normal subgroup $\Delta$ 
such that $\bar{\Ga}\coloneq\Ga/\Delta,\Delta\in\UE(k)$. 
Let a finite unital symmetric subset $F\Subset\Ga$ 
be given. 
We write $\bar{F}$ for the image of $F$ in $\bar{\Ga}$. 
Since $\bar{\Ga}\in\UE(k)$, 
for $k \coloneq k(|F|)$ and $\ve \coloneq |F|^{-1}$, 
there is $\eta\colon \bar{\Ga} \to \Prob \bar{\Ga}$ 
such that 
\[
\bigcup_{p \in \bar{\Ga}} \supp \eta^p\subset \bar{F}^k
\ \mbox{ and }\ 
\max_{s\in\bar{F}}\sup_{p\in\bar{\Ga}} 
 \| \eta^{sp} - s\eta^p \|_1\le\ve.
\]
Take a lift $\varsigma\colon\bar{\Ga}\to\Ga$ 
that satisfies $\varsigma(\bar{F}^l) \subset F^l$ for every $l$ 
and set the cocycle 
$\alpha\colon \Ga\times\bar{\Ga} \to \Delta$ by 
$\alpha(s,p)\coloneq\varsigma(\bar{s}p)^{-1}s\varsigma(p)$.
One has $\alpha(F,\bar{F}^k)\subset F^{2(k+1)}\cap\Delta\eqcolon F'$. 
Since $\Delta\in\UE(k)$, 
for $k' \coloneq k(|F'|)$ and $\ve'\coloneq|F'|^{-1}$, 
there is 
$\xi\colon \Delta\to\Prob\Delta$ such that 
\[
\bigcup_{x \in \Delta} \supp \xi^x\subset (F')^{k'}
\ \mbox{ and }\ 
\max_{s\in F'}\sup_{x\in\Ga} 
 \| \xi^{sx} - s\eta^x \|_1\le\ve'.
\]
Define $\zeta\colon \Ga\to\Prob\Ga$ by 
$\zeta^x\coloneq \sum_{p\in\bar{\Ga}}\eta^{\bar{x}}(p)\varsigma(p)\xi^{\varsigma(p)^{-1}x}$.
One has
$\supp\zeta^x\subset F^{k+2(k+1)k'}$ for all $x\in\Ga$ and,  
for $s\in F$ and $x\in\Ga$, 
\begin{align*}
\zeta^{s^{-1}x}
 &= \sum_{p\in\bar{\Ga}}\eta^{\bar{s}^{-1}\bar{x}}(p)
   \varsigma(p)\xi^{\alpha(s,p)^{-1}\varsigma(\bar{s}p)^{-1}x}\\
 &\approx_{\ve'} \sum_{p\in\bar{\Ga}}\eta^{\bar{s}^{-1}\bar{x}}(p)
   s^{-1}\varsigma(\bar{s}p)\xi^{\varsigma(\bar{s}p)^{-1}x}\\
 &= \sum_{p\in\bar{\Ga}}(\bar{s}\eta^{\bar{s}^{-1}\bar{x}})(p)
   s^{-1}\varsigma(p)\xi^{\varsigma(p)^{-1}x}
  \approx_\ve s^{-1}\zeta^x.
\end{align*}
This proves $\Ga\in\UE(k_1)$ for some $k_1$ that is computable from $k$. 
\end{proof}

\end{document}